\documentclass{amsart}
\usepackage{amsmath, amssymb, amsthm, amsfonts, amstext}
\usepackage{xcolor}
\usepackage[all]{xy}
\usepackage[colorlinks=true, citecolor=red, linkcolor=blue]{hyperref}
\usepackage[T1]{fontenc}
\usepackage[ansinew]{inputenc}
\usepackage{xcolor}
\usepackage{enumerate,stmaryrd}
\usepackage[british,UKenglish,USenglish,english,american]{babel}

\newtheorem{theorem}{Theorem}
\theoremstyle{plain}

\newtheorem{corollary}{Corollary}

\newtheorem{lemma}{Lemma}

\newtheorem{proposition}{Proposition}
\newtheorem{remark}{Remark}

\numberwithin{equation}{section}

\begin{document}
\title{Weak Commutativity Between two Isomorphic Polycyclic Groups.}
\author{Bruno C\' esar R. Lima}

\address{Secretaria de Estado Educa\c c\~ao do Distrito Federal, Bras\' ilia, DF, Brazil}
\email{bruno\_crlima@hotmail.com}
\author{Ricardo N. Oliveira }
\address{Universidade Federal de Goi\'as, Goi\^ania, Goi\'as, Brazil}
\email{ricardo@ufg.br}
\date{\today}
\thanks{We are grateful to Professor S. Sidki for suggesting  the problem and for his support and encouragement.}
\thanks{The first author acknowledges support from PROCAD-CAPES for Sandwich Doctorate studies.} 
\thanks{The second author acknowledges support from PROCAD-CAPES for post-doctoral studies.}

\begin{abstract}
\hspace{0,09cm} The operator of weak commutativity between isomorphic groups $H$ and $H^{\psi }$ was defined by Sidki as 
\begin{equation*}
\chi (H)=\left\langle H\,H^{\psi }\mid \lbrack h,h^{\psi }]=1\,\forall
\,h\in H\right\rangle \text{.}
\end{equation*}%
It is known that the operator $\chi $ preserves group properties such as
finiteness, solubility and also nilpotency for finitely generated groups. We
prove in this work that $\chi $ preserves the properties of being polycyclic
and polycyclic by finite. As a consequence of this result, we conclude that
the non-abelian tensor square $H\otimes H$ of a group $H$, defined by Brown
and Loday, preserves the property polycyclic by finite. This last result
extends that of Blyth and Morse who proved that $H\otimes H$ is polycyclic
if $H$ is polycyclic.
\end{abstract}
\maketitle

\section{Introduction}

Given two groups $H$ and $H^{\psi }$ which are isomorphic via $\psi
:h\rightarrow h^{\psi }$, the following group construction was introduced
and analysed in \cite{sidki-wp} 
\begin{equation*}
\chi (H)=\left\langle H\,,H^{\psi }\mid \lbrack h,h^{\psi }]=1\,\forall
\,h\in H\right\rangle \text{.}
\end{equation*}

The weak commutativity group $\chi (H)$ maps onto $H$ by $h\rightarrow h$, $h^{\psi }\rightarrow h$
with kernel $L(H)=\left\langle h^{-1}h^{\psi },h\;\in H\right\rangle $ and
maps onto $H\times H$ by $h\rightarrow \left( h,1\right) ,h^{\psi
}\rightarrow \left( 1,h\right) $ with kernel $D(H)=[H,H^{\psi }]$. It is an
important fact that $L(H)$ and $D(H)$ commute. Let $T(H)$ be defined as the
subgroup of $H\times H\times H$ generated by $\{(h,h,1),(1,h,h)\mid h\in H\}$%
. Then $\chi (H)$ maps onto $T(H)$ by $h\rightarrow \left( h,h,1\right)$, $h^{\psi }\rightarrow \left( 1,h,h\right) $, with kernel $W(H)=L(H)\cap D(H)$%
, an abelian group. A further normal subgroup of $\chi (H)$ is $R(H)={%
[H,L(H),H^{\psi }]}$ where the quotient $\frac{W(H)}{R(H)}$ is isomorphic to
the Schur Multiplier $M(H)$. Proving properties of $\chi (H)$ depend
crucially upon understanding $R(H)$. Thus, if $H$ is polycyclic by finite
then so is $\frac{\chi (H)}{R(H)}$ and to prove more generally that $\chi
(H)$ is polycyclic by finite depends upon showing that $R(H)$ is finitely
generated.\ 

We will make use of the following facts from \cite[p. 201]{sidki-wp},

\begin{lemma}
\label{lema:R}\ The groups $D(H)$, $L(H),W(H)$ and $R(H)=[H,L(H),H^{\Psi }]$
satisfy:
\begin{enumerate}
\item[(i)] $R(H)\lhd \chi (H)$, $R(H)$ is $\Psi $-invariant;
\item[(ii)] $[W(H),H]\leq R(H)\leq W(H)\leq D(H);$
\item[(iii)] $[h_{1},h_{2}^{\psi }]^{h_{3}}=[h_{1},h_{2}^{\psi }]^{h_{3}^{\Psi }}$%
for any $h_{1},\;h_{2},\;h_{3}\;\in H.$
\item[(iv)] $D(H)$ centralizes $L(H);$
\item[(v)] $[h_{1},h_{2}^{\psi }]^{h_{3}}=[{h_{1}}^{h_{3}},({h_{2}}^{h_{3}})^{\psi
}]$ holds in in $\chi (H)$ mod $R(H)$, for any elements $h_1, h_2, h_3 \in H$.
\end{enumerate}
\end{lemma}

We have the following diagram of subgroups of $\chi (H)$   
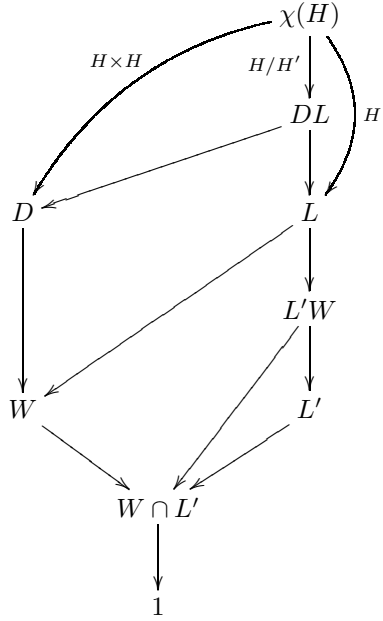
\begin{equation*}
\xymatrix{ & & \chi(H) \ar[d]_{H/H'} \ar@/_0.6cm/[ddll]_{H\times H}
\ar@/^0.6cm/[dd]^{H} \\ & & {DL} \ar[dll] \ar[d] \\ {D} \ar[dd] & & {L}
\ar[d] \ar[ddll] \\ & & {L'W} \ar[d] \ar[ldd] \\ {W} \ar[rd]& & L' \ar[ld]\\
& W\cap L' \ar[d] & \\ & 1 & }
\end{equation*}

Our main result in this paper is

\begin{theorem}
Let $H$ be a group which is polycyclic by finite. Then the group $\chi (H)$
and the non-commutative tensor square $H\otimes H$ are also polycyclic by
finite.
\end{theorem}

In last subsection, we discuss different behaviours of $R\left( H\right) $,
specially for certain polycyclic groups $H$.

\section{Connection with the augmentation ideal of $\mathbb{Z}(H)$}

Given a group $H$ with identity $1$, consider the group ring $\mathbb{Z}(H)$
and its augmentation ideal ${\mathcal{A}}_{\mathbb{Z}}(H)$ which is
generated by $\{h-1\mid h\in H\}$. Consider also $I_{2}(H)$ the ideal of $%
\mathbb{Z}(H)$ generated by $\{(h-1)^{2}\mid h\in H\}$. Let $\tilde{G}=%
\mathbb{Z}(H)\rtimes H$ be semidirect product of $\mathbb{Z}(H)$ by $H$,
where $\mathbb{Z}(H)$ is written additively and conjugation of $\mathbb{Z}%
(H) $ by $H$ is described as right multiplication 
\begin{equation*}
(\sum {(x_{h}h,1)})^{(0,h_{1})}=\sum {(x_{h}hh_{1},1)}\text{.}
\end{equation*}%
In addition, let $G$ be the subgroup of $\tilde{G}$ generated by ${\mathcal{A%
}}_{\mathbb{Z}}(H)\times \{1\}$ and $\{0\}\times H$. Then, $N=I_{2}(H)\times
\{1\}$ is a normal subgroup of $\tilde{G}$.

\begin{proposition}
\label{prop:L/L'fg} Let $L=L(H)$ and $L^{\prime }$ be its derived subgroup.
Then the application 
\begin{equation*}
\epsilon :H\cup H^{\psi }\rightarrow \frac{G}{N}
\end{equation*}%
defined by $\epsilon :h\rightarrow N(0,h)$ and $h^{\psi }\rightarrow
N(0,h)^{(1,1)},\;\;\forall h\;\in H$, extends to an epimorphism 
\begin{equation*}
\epsilon :\chi (H)\rightarrow \frac{G}{N}
\end{equation*}%
with kernel $L^{\prime }$. Furthermore, if $H$ is a finitely generated group
then so is the quotient group $\frac{L}{L^{\prime }}$.

\begin{proof}
Since, for $h\in H$ 
\begin{eqnarray}
\lbrack (0,h),(1,1)] &=&(0,h)^{-1}(0,h)^{(1,1)}  \notag  \label{eq:com_hu} \\
&=&(1,1)^{(0,-h)}(1,1)  \notag \\
&=&({-h+1},1)\in {\mathcal{A}}_{\mathbb{Z}}(H)\times \{1\}\text{,}
\end{eqnarray}%
we obtain 
\begin{equation*}
{\mathcal{A}}_{\mathbb{Z}}(H)\times \{1\}=[\{0\}\times H,(1,1)]\text{ }
\end{equation*}%
and%
\begin{eqnarray*}
G &=&\left\langle \{1\}\times H,\;(\{1\}\times H)^{(1,1)}\right\rangle \\
&\cong &\left\langle H,\;H^{\phi }\mid \;[H,\phi ]^{^{\prime
}}=1\right\rangle
\end{eqnarray*}%
by \cite[p. 189]{sidki-wp}. Given $h$ $\in H$ we conclude that in $G$, 
\begin{eqnarray*}
\lbrack (0,h),(0,h)^{(1,1)}] &=&1\Leftrightarrow \lbrack
(1,1),(0,h),(0,h)]=1, \\
\lbrack (1,1),(0,h),(0,h)] &=&((h-1)^{2},1)\text{.}
\end{eqnarray*}%
Therefore 
\begin{equation*}
\epsilon :h\rightarrow N(0,h),h^{\psi }\rightarrow N(0,h)^{(1,1)}
\end{equation*}%
extends to an epimorphism $\epsilon :\chi (H)\rightarrow \frac{G}{N}$ and $%
Ker_{\epsilon }=L^{\prime }$.

As 
\begin{equation*}
\chi (H)=L\cdot H\text{ and }L^{\epsilon }=\dfrac{{\mathcal{A}}_{\mathbb{Z}%
}(H)\times \{1\}}{N}
\end{equation*}%
it follows that 
\begin{equation*}
\frac{L}{L^{\prime }}\cong \dfrac{{\mathcal{A}}_{\mathbb{Z}}(H)}{I_{2}(H)}%
\text{.}
\end{equation*}

Let $S=\{a_{1},\ldots ,a_{n}\}$ be a generating set for $H$. The following
equations hold in ${\mathcal{A}}_{\mathbb{Z}}(H)$ modulo $I_{2}(H)$ for all $%
a_{i},\;a_{j}\in S$: 
\begin{eqnarray*}
(a_{i}-1)^{2} &=&0,\text{ }a_{i}^{2}=2a_{i}-1, \\
a_{i}^{k} &=&ka_{i}-(k-1)\;\forall \;k\in \mathbb{Z}; \\
(a_{j}a_{i})^{-1} &=&2-a_{j}a_{i}, \\
a_{i}^{-1}a_{j}^{-1} &=&(2-a_{i})(2-a_{j})=a_{i}a_{j}-2a_{i}-2a_{j}+4, \\
a_{j}a_{i} &=&-a_{i}a_{j}+2a_{j}+2a_{i}-2\text{.}
\end{eqnarray*}

We conclude that $\left\{ I_{2}(H)+a_{i_{1}}a_{i_{2}}\cdots a_{i_{s}}\mid
1\leq i_{1}<i_{2}<\ldots i_{s}\leq n\right\} $ generates the abelian group $%
\frac{\mathbb{Z}(H)}{I_{2}(H)}$ and in particular, $\frac{{\mathcal{A}}(H)}{%
I_{2}(H)}$ is a finitely generated.
\end{proof}
\end{proposition}



\begin{lemma}
(\cite{stambach}) Let $H$ a group finitely presented. Then the Schur
Multiplier of $H$ is finitely generated.
\end{lemma}

\begin{theorem}
\label{teo:chipoly} If $H$ be a polycyclic (polycyclic by finite) group then
so is $\chi (H)$.

\begin{proof}
We will prove the assertion for $H$ polycyclic by finite; the proof for $H$
polycyclic is similar. Denote $L\left( H\right) =L,D\left( H\right) =D$. It
follows directly from the above diagram of subgroups that 
\begin{equation*}
\frac{\chi (H)}{D}\cong H\times H\text{, }\frac{H}{H^{\prime }}\cong \frac{%
\chi (H)}{DL}\cong \frac{\frac{\chi (H)}{D}}{\frac{DL}{D}}\text{\quad \quad }
\end{equation*}%
and 
\begin{equation*}
\frac{DL}{D}\cong \frac{L}{W}\cong H^{\prime }\cdot H
\end{equation*}%
are all polycyclic by finite. Since $\frac{L}{L^{\prime }},\frac{L^{\prime }W%
}{W}\left( \cong \frac{L^{\prime }}{L^{\prime }\cap W}\right) $ are finitely
generated groups, we conclude that $\frac{L}{L^{\prime }\cap W}$ is
polycyclic by finite. Since $W=D\cap L$ and $\left[ D,L\right] =1$, it
follows that $L^{\prime }\cap W\leq L^{\prime }\cap Z(L)$. Therefore, by a
theorem of Schur \cite[p. 19]{Karpilovsky}, the group $L^{\prime }\cap W$ is
isomorphic to a subgroup of the Schur Multiplier of $\frac{L}{L^{\prime
}\cap W}$. Since $\frac{L}{L^{\prime }\cap W}$ is polycyclic by finite, by
the above lemma, it follows that $M(\frac{L}{L^{\prime }\cap W})$ is
finitely generated. Therefore, $L^{\prime }\cap W$ is finitely generated and
consequently, $L$ is polycyclic by finite and finally so is $\chi (H)$.
\end{proof}
\end{theorem}
\newpage
\section{Connection with the non-commutative tensor square}

We recall the non-commutative tensor square introduced by Brown-Loday
\cite{BL} 
$$
H\otimes H=\left< h_{1}\otimes h_{2} \right.  \mid h_{1}h_{2}\otimes h_{3}=({h_{1}}%
^{h_{2}}\otimes {h_{3}}^{h_{2}})(h_{2}\otimes h_{3}),
$$
$$\hspace{3,1cm}
h_{1}\otimes h_{2}h_{3}=(h_{1}\otimes h_{3})({h_{1}}^{h_{2}}\otimes {h_{3}}%
^{h_{2}})\;\forall \;h_{1},\;h_{2},\;h_{3}\in H\left> \right.
$$
and the group defined by N. Rocco \cite{rocco-crnts} 
$$
\nu (H)=\left\langle H,H^{\psi }\mid \lbrack h_{1},h_{2}^{\psi
}]^{h_{3}^{\psi }}=[h_{1},h_{2}^{\psi
}]^{h_{3}}=[h_{1}^{h_{3}},(h_{2}^{h_{3}})^{\psi }], \; \forall h_1,\; h_2,\;
h_3 \in H\right\rangle \text{.}
$$

Rocco showed in \cite{rocco-crnts} that there exist an isomorphism between
the subgroup $\tau (H)=[H,H^{\psi }]$ of $\nu (H)$ and $H\otimes H$.
Moreover Brown-Loday \cite{BL} showed that for $J(H)$, is the kernel of the
epimorphism $\tau (H)\rightarrow H^{\prime }$ defined by $[h_{1},h_{2}^{\psi
}]\mapsto \lbrack h_{1},h_{2}]$, and for $\Delta (H)=\left\langle [h,h^{\psi
}]\mid h\in H\right\rangle $, then the Schur Multiplier $M(H)$ is isomorphic
the quotient $J(H)/\Delta (H)$.

It was shown in \cite{bm} that if $H$ is polycyclic then so is $\nu (H)$.
Thus, the following corollary generalizes this result.

\begin{corollary}
Let $H$ be a polycyclic by finite group. Then $\nu (H)$ and $H\otimes H$ are
polycyclic by finite groups.
\begin{proof}
By \cite[p. 68-69]{rocco-crnts} we have that $\frac{\chi (H)}{R(H)}\cong 
\frac{\nu (H)}{\Delta (H)}$, where the diagonal group $\Delta (H)=\left\langle [h,h^{\psi
}]\mid h\in H\right\rangle $ is a normal finitely generated abelian subgroup
of $\nu(H)$ and $R(H)=[H,L,H^{\psi }]$ is a subgroup of $\chi (H)$.
\end{proof}
\end{corollary}

\section{The subgroup $R(H)$ for a polycyclic group $H$}

\begin{proposition}
\label{prop:gerR} Let $H$ be a group and $T$ a transversal for $H^{\prime }\ 
$in $H$. Then, 
\begin{equation*}
R(H)={\left\langle [h_{1},h_{2}^{\psi }]^{h_{3}}[{h_{1}}^{h_{3}},({h_{2}}%
^{h_{3}})^{\psi }]^{-1}\mid h_{1},\;h_{2},\;h_{3}\;\in H\right\rangle }^{T}.
\end{equation*}
\end{proposition}

\begin{proof}
First let's show that 
\begin{equation*}
R(H)={\left\langle [h_{1},h_{2}^{\psi }]^{h_{3}}[{h_{1}}^{h_{3}},({h_{2}}%
^{h_{3}})^{\psi }]^{-1}\mid h_{1},\;h_{2},\;h_{3}\;\in H\right\rangle }%
^{\chi (H)}.
\end{equation*}%
The subgroup of $\chi (H)$,   
\begin{equation*}
J(H)={\left\langle [h_{1},h_{2}^{\psi }]^{h_{3}}[{h_{1}}^{h_{3}},({h_{2}}%
^{h_{3}})^{\psi }]^{-1}\mid h_{1},\;h_{2},\;h_{3}\;\in H\right\rangle }%
^{\chi (H)}
\end{equation*}%
is contained in $R(H)$, by \hyperref[lema:R]{ Lemma \ref*{lema:R} (v)}. 

By Rocco \cite[p. 68-69]{rocco-crnts} the epimorphism $\varepsilon :\chi
(H)\rightarrow \frac{\nu (H)}{\Delta (H)}$ given by $h\mapsto \Delta (H)h$, $%
h^{\psi }\mapsto \Delta (H)h^{\psi },\;\forall h\in H$, has kernel $R(H)$.
So $\varepsilon $ induces $\bar{\varepsilon}:\frac{\chi (H)}{J(H)}%
\rightarrow \frac{\nu (H)}{\Delta (H)}$. On the other hand, the application $%
\varsigma :\frac{\nu (H)}{\Delta (H)}\rightarrow \frac{\chi (H)}{J(H)}$ such
that $\Delta (H)h\mapsto J(H)h$, $\Delta (H)h^{\psi }\mapsto J(H)h^{\psi }$
extend to a epimorphism $\bar{\varsigma}:\frac{\nu (H)}{\Delta (H)}%
\rightarrow \frac{\chi (H)}{J(H)}$, because \newpage
$$
\frac{\nu (H)}{\Delta (H)}=\left< H,H^{\psi }\mid \lbrack h,h^{\psi
}]=1\;,\;[h_{1},h_{2}^{\psi }]^{h_{3}}=[{h_{1}}^{h_{3}},({h_{2}}%
^{h_{3}})^{\psi }]=[h_{1},h_{2}^{\psi }]^{h_{3}^{\psi }},\right.
$$
$
\hspace{3,2cm}\left. \forall h,\,h_{1}\,h_{2},\,h_{3}\,\in H\right>
$\\
and the relations 
\begin{equation*}
\lbrack h_{1},h_{2}^{\psi }]^{h_{3}^{\psi }}=[h_{1},h_{2}^{\psi }]^{h_{3}}=[{%
h_{1}}^{h_{3}},(h_{2}^{h_{3}})^{\psi }],\;\;[h,h^{\psi }]=1,\;\forall
h,\;h_{1},\;h_{2},\;h_{3}\in H
\end{equation*}%
holds in $\frac{\chi (H)}{J(H)}$. How $\bar{\varsigma}\bar{\varepsilon}$ is
the identity in $\frac{\chi (H)}{J(H)}$, follows that $R(H)=J(H)$.

Now we have that $R(H)\leq D(H)$, so by \hyperref[lema:R]{ Lemma \ref*{lema:R} (iii)} , it follows that

\begin{equation*}
R(H)={\left\langle [h_{1},h_{2}^{\psi }]^{h_{3}}[{h_{1}}^{h_{3}},({h_{2}}%
^{h_{3}})^{\psi }]^{-1}\mid h_{1},\;h_{2},\;h_{3}\;\in H\right\rangle }^{H}.
\end{equation*}

Again, by \hyperref[lema:R]{ Lemma \ref*{lema:R} (iii) and (iv)}, we have 
\begin{equation*}
\lbrack R(H),H^{\prime \psi }]=1.
\end{equation*}%
Then 
\begin{equation*}
R(H)={\left\langle [h_{1},h_{2}^{\psi }]^{h_{3}}[{h_{1}}^{h_{3}},({h_{2}}%
^{h_{3}})^{\psi }]^{-1}\mid h_{1},\;h_{2},\;h_{3}\;\in H\right\rangle }^{T}.
\end{equation*}
\end{proof}

\begin{proposition}
\label{prop:R} Let $H$ be a polycyclic group with a polycyclic generators $%
S=\{a_{1},a_{2},\ldots a_{n}\}$ and let $T$ be a transversal for $H^{\prime }
$ in $H$. Then 
\begin{equation*}
R(H)={\left\langle [a_{i},a_{j}^{\psi }]^{a_{k}}[{a_{i}}^{a_{k}},({a_{j}}%
^{a_{k}})^{\psi }]^{-1}\mid a_{i},\;a_{j},\;a_{k}\,\in S\right\rangle }^{T}.
\end{equation*}
\end{proposition}

\begin{proof}
Let $S=\{a_{1},\ldots ,a_{n}\}$ be a generating set for $H$ then by \cite[p. 37]{mcdermott}, the subgroup 
\begin{equation*}
K={\left\langle [h_{1},h_{2}^{\psi }]^{h_{3}}[{h_{1}}^{h_{3}},({h_{2}}%
^{h_{3}})^{\psi }]^{-1}\mid h_{1},\;h_{2},\;h_{3}\,\in H\right\rangle }%
^{H\ast H^{\psi }}
\end{equation*}%
of the free product $H\ast H^{\psi }$ has the presentation 
\begin{equation*}
J={\left\langle [a_{i},a_{j}^{\psi }]^{a_{k}}[{a_{i}}^{a_{k}},({a_{j}}%
^{a_{k}})^{\psi }]^{-1}\mid a_{i},\;a_{j},\;a_{k}\,\in S\right\rangle }%
^{H\ast H^{\psi }}.
\end{equation*}%
If $\phi :H\ast H^{\psi }\mapsto \chi (H)$ \'{ }is the natural epimorphism,
we conclude that

\begin{eqnarray*}
\phi (J)=R(H) &=&{\left\langle [a_{i},a_{j}^{\psi }]^{a_{k}}[{a_{i}}%
^{a_{k}},({a_{j}}^{a_{k}})^{\psi }]^{-1}\mid a_{i},\;a_{j},\;a_{k}\,\in
S\right\rangle }^{\chi (H)} \\
&\overset{\hyperref[prop:gerR]{Prop\ref*{prop:gerR}}}{=}&{\left\langle
[a_{i},a_{j}^{\psi }]^{a_{k}}[{a_{i}}^{a_{k}},({a_{j}}^{a_{k}})^{\psi
}]^{-1}\mid a_{i},\;a_{j},\;a_{k}\,\in S\right\rangle }^{T}.
\end{eqnarray*}
\end{proof}

\subsection{Different behaviours of $R\left( H\right) $}

It is difficult in general to obtain information about $R\left( H\right) $
and there are a few cases for which it is described. The following remark is
helpful in establishing the non-triviality of $R\left( H\right) $.

\begin{remark}
Let $H,K$ be groups and $\varphi :H\rightarrow K$ be an epimorphism. Then $%
\varphi $ extends to an epimorphism $\widehat{\varphi }:\chi (H)\rightarrow
\chi (K)$ by $\widehat{\varphi }:h\rightarrow h^{\varphi },h^{\psi
}\rightarrow \left( h^{\varphi }\right) ^{^{\psi }}$ (that is, by having $%
\widehat{\varphi }$ commute with $\psi $). Therefore, $L\left( H\right) ^{%
\widehat{\varphi }}=L\left( K\right) $ and $R\left( H\right) ^{\widehat{%
\varphi }}=R\left( K\right) $. In particular, $R\left( H\right) $ is
non-trivial provided $R\left( K\right) $ is non-trivial.
\end{remark}

Information about $R\left( H\right) $ is known for finitely generated
abelian groups $H$ (see Section 4.2 of \cite{sidki-wp}).

\begin{theorem}
(Theorem 4.2.1 \cite[p. 204]{sidki-wp}) Let $H$ be an abelian group. Then,
\begin{enumerate}
\item[(i)] $D=W=\left[ L,H\right] ,$ $R=\left[ D,H\right] =\left[ L,2H\right] $;
\item[(ii)] $L$ is nilpotent of class $\leq 2$, $L^{\prime }\leq D\leq Z\left(
L\right) ,$ $L^{\prime }\leq Z\left( \chi (H)\right) $; 
\item[(iii)] $L^{\prime
}=D^{2}=\left[ H^{2},H^{\psi }\right] ,$ $R^{2}=1$.
\end{enumerate}

\end{theorem}

In case $H$ is a finite elementary abelian $2$-group, we have

\begin{proposition}
Let $H$ be an elementary abelian $2$-groups of rank $k$ and order $n$ ($%
=2^{k}$). Then, 
\begin{enumerate}
\item[(i)] $\chi (H)$ is isomorphic to the natural extension of ${%
\mathcal{A}}_{\mathbb{Z}}(H)$ by $H$ of order $2^{n-1}n$, where ${\mathcal{A}%
}_{\mathbb{Z}}(H)$ corresponds to $L\left( H\right) $; 
\item[(ii)] the derived
subgroup $\gamma _{2}\left( \chi (H)\right) =D\left( H\right) $; 
\item[(iii)] $%
\gamma _{3}\left( \chi (H)\right) =R\left( H\right) $ and has order $%
2^{n-1-k-\binom{k}{2}}$.
\end{enumerate}

\begin{proof}
The first two items follow directly from the material in Section 4.2 of \cite%
{sidki-wp}. The third item follows from 
\begin{eqnarray*}
\chi (H)^{\prime } &=&[L\left( H\right) ,H], \\
\lbrack L\left( H\right) ,H,\chi (H)] &=&[L\left( H\right) ,H,H]=[L\left(
H\right) ,H,H^{\psi }]=R\left( H\right)
\end{eqnarray*}%
and from 
\begin{equation*}
\left\vert \frac{\chi (H)}{R\left( H\right) }\right\vert =\left\vert
H\right\vert ^{2}\left\vert M\left( H\right) \right\vert ,\left\vert M\left(
H\right) \right\vert =2^{\binom{k}{2}}\text{.}
\end{equation*}
\end{proof}
\end{proposition}

In view of the above remark, we conclude

\begin{proposition}
Let $H$ be a group which has as homomorphic image an elementary abelian $2$%
-group of rank at least $3$. Then $R\left( H\right) $ is non-trivial.
\end{proposition}

On the other hand, $R(H)$ can be trivial, as is the case of $H$ a perfect
group (see Section 4.4 of \cite{sidki-wp}). For polycyclic groups, we have

\begin{proposition}
The group $R\left( H\right) $ is trivial for the following polycyclic
groups: 
\begin{enumerate}
\item[(i)] $H$ a finite abelian $p$-group for $p$ odd; 
\item[(ii)] $H$ a $2$%
-generated free nilpotent group of class $2$; 
\item[(iii)] $H$ a metacyclic group.
\end{enumerate}

\end{proposition}

\begin{proof}
\begin{enumerate}

\item[(i)] This case was shown in Theorem 4.2.4 of \cite{sidki-wp}.

\item[(ii)] Let $S=\left\{ a_{1},a_{2}\right\} $ be a generating set for $H$. Then $%
H/H^{^{\prime }}=\left\langle H^{^{\prime }}a_{1},H^{^{\prime
}}a_{2}\right\rangle $ and $\ H^{^{\prime }}=\left\langle
[a_{1},a_{2}]\right\rangle $.

Let $\ K=\left\langle [h_{1},h_{2}^{\Psi
}]^{h_{3}}[h_{1}^{h_{3}},(h_{2}^{h_{3}})^{\Psi }]^{-1}|h_{1},h_{2},h_{3}\in
H\right\rangle ^{H\ast H^{\Psi }}$ in the group $H\ast H^{\Psi }$. By \cite[p. 37]{mcdermott} we can simplify the generating set of $K$ to 
\begin{equation*}
K=\left\langle [a_{1},a_{2}^{\Psi }]^{b}[a_{1}^{b},(a_{2}^{b})^{\Psi
}]^{-1}\mid b\in S\cup \{[a_{1},a_{2}]\}\right\rangle ^{H\ast H^{\Psi }}.
\end{equation*}%
Taking the natural epimorphism $\phi :H\ast H^{\Psi }\rightarrow \chi (H)$ ,
it is easy see that 
\begin{equation*}
R(H)=\phi (K)=\left\langle [a_{1},a_{2}^{\Psi
}]^{b}[a_{1}^{b},(a_{2}^{b})^{\Psi }]^{-1}\mid b\in
\{a_{1},a_{2},[a_{1},a_{2}]\}\right\rangle ^{\chi (H)}.
\end{equation*}%
Now, the following relations hold in $\chi (H)$,%
\begin{eqnarray*}
\lbrack a_{1},a_{2}^{\Psi }]^{a_{1}} &=&[a_{1}^{\Psi
},a_{2}]^{a_{1}}=[a_{1}^{\Psi },a_{2}^{a_{1}}]\text{,} \\
\lbrack a_{1},a_{2}^{\Psi }]^{a_{2}} &=&[a_{1}^{a_{2}},a_{2}^{\Psi }]\text{.}
\end{eqnarray*}%
Since $\gamma _{3}(H)=1$, it follows that  
\begin{equation*}
\lbrack a_{1},a_{2}{}^{\Psi }]^{[a_{1},a_{2}]}=[a_{1},a_{2}{}^{\Psi
}]^{[a_{1},a_{2}{}^{\Psi }]}=[a_{1},a_{2}{}^{\Psi }]\text{..}
\end{equation*}%
\textbf{\ }

Thus, $R(H)$ is trivial.

\item[(iii)] Let $S=\{a_{1},\,a_{2}\}$ be a polycyclic generators for $H$. By 
\hyperref[prop:R]{Proposition \ref*{prop:R}} we can take 
\begin{equation*}
R(H)={\left\langle [a_{i},a_{j}^{\psi }]^{a_{k}}[{a_{i}}^{a_{k}},({a_{j}}%
^{a_{k}})^{\psi }]^{-1}\mid a_{i},\;a_{j},\;a_{k}\,\in S\right\rangle }^{T},
\end{equation*}%
were $T$ is any transversal for $H/H^{\prime }$. In the group $\chi (H)$ the
following relations hold, 
\begin{eqnarray*}
\lbrack a_{1},{a_{2}}^{\psi }]^{a_{1}} &=&[{a_{1}}^{\psi },a_{2}]^{a_{1}}=[{%
a_{1}}^{\psi },a_{2}^{a_{1}}] \\
&=&[{a_{1}},(a_{2}^{a_{1}})^{\psi }], \\
\lbrack a_{1},{a_{2}}^{\psi }]^{a_{2}} &=&[{a_{1}}^{a_{2}},(a_{2})^{\psi }]%
\text{.}
\end{eqnarray*}%
Thus, $R(H)$ is trivial.

\end{enumerate}
\end{proof}

\end{document}